\documentclass[11pt]{article}
\usepackage{amsmath,amssymb,amsfonts,amsthm,enumerate,xcolor}
\usepackage[normalem]{ulem} %

\usepackage[top=1in, bottom=1in, left=1.25in, right=1.25in, marginparwidth=23mm, marginparsep=2mm]{geometry}

\usepackage[unicode,breaklinks=true,colorlinks=true]{hyperref}

\newtheorem{thm}{Theorem}[section]

\newtheorem{lem}[thm]{Lemma}

\theoremstyle{definition}

\numberwithin{equation}{section}
\newcommand{\bke}[1]{\left ( #1 \right )}

\newcommand{\norm}[1]{\left \| #1 \right \|}

\newcommand{\R}{\mathbb{R}}
\newcommand{\N}{\mathbb{N}}
\renewcommand{\div}{\mathop{\rm div}\nolimits}

\newcommand{\pd}{\partial}

\newcommand\Om{\Omega}
\newcommand{\si}{\sigma}
\newcommand\De{\Delta}

\newcommand{\nb}{\nabla}
\newcommand{\lec}{{\ \lesssim \ }}

\newcommand{\bka}[1]{{\langle #1 \rangle}}

\newcommand\al{\alpha}

\newcommand\e{\epsilon}

\renewcommand\th{\theta}

\newcommand\la{\lambda}

\newcommand\Th{\Theta}

\newcommand{\NN}{\mathbb{N}}

\newcommand{\cR}{\mathcal{R}}

\newcommand{\EQ}[1]{\begin{equation} #1 \end{equation}}
\newcommand{\EQS}[1]{\begin{equation}\begin{split} #1 \end{split}\end{equation}}

\newcommand{\EQN}[1]{\begin{equation*}\begin{split} #1 \end{split}\end{equation*}}

\newcommand{\loc}{{\mathrm{loc}}}
\newcommand{\tsum}{\textstyle \sum}
\newcommand{\all}{{\mathrm{all}}}

\begin{document}
\title{Large discretely self-similar solutions
to Oberbeck-Boussinesq system with Newtonian gravitational field}%
\author{Tai-Peng Tsai}

\date{}

\maketitle
\begin{abstract}
Discretely self-similar solutions to Oberbeck-Boussinesq system with Newtonian gravitational field for large discretely self-similar initial data are constructed in this note, extending the construction of Brandolese and Karch (arXiv:2311.01093) on self-similar solutions. It follows the approach of Bradshaw and Tsai (Ann.~Henri Poincar\'e 2017) and find an explicit a priori bound for the deviation from suitably revised profiles in similarity variables.

\medskip

\emph{keywords:} Oberbeck-Boussinesq system, Navier-Stokes equations, self-similar solutions, discretely self-similar solutions, similarity variable, revised profile.

\medskip

\end{abstract}

\section{Introduction}

The Oberbeck-Boussinesq system is a mathematical model that describes a Newtonian incompressible fluid with buoyancy force caused by the variation of fluid temperature from its equilibrium value. It reads
\EQS{\label{OB1}
\partial_t v-\Delta v +v\cdot\nabla v+\nabla p &= \th \nb G+f,\\
\nabla \cdot v&=0,\\
\partial_t \th-\Delta \th +v\cdot\nabla \th& = 0,
}
with the unknown fluid velocity $v(x,t)$, the temperature $\th(x,t)$ and the pressure $p = p(x,t)$, defined for $x \in \R^3$ and $ t>0$. Here $ G(x)$ stands for the gravitational potential, and we assume $G(x)=|x|^{-1}$ in this paper. Finally $f(x,t)$ stands for a given external force.
The system \eqref{OB1}
is coupled with initial conditions
\EQ{\label{OB2}
v(x,0)=v_0(x), \quad \th(x,0)=\th_0(x).
}

The system \eqref{OB1} enjoys the following scaling property: If $(v,\th,p)$ is a solution for \eqref{OB1} with force $f$, then 
\[
v^\la(x,t)= \la v(\la x, \la^2 t), \quad
\th^\la(x,t)= \la \th(\la x, \la^2 t), \quad
p^\la(x,t)= \la^2 p(\la x, \la^2 t), 
\]
is also a solution of the same system with force $f^\la(x,t)= \la^3 f(\la x, \la^2 t)$.  Assuming
\EQ{\label{self-similar}
v^\la = v, \quad \th^\la=\th, \quad p^\la = p, \quad f^\la=f
}
for each $\la>0$, we say $(v,\th)$ is a \emph{self-similar} solution. If \eqref{self-similar} is valid for one particular $\la>1$, we say $(v,\th)$ is \emph{discretely self-similar} with factor $\la$, or $\la$-DSS.

The study of large self-similar and discretely self-similar solutions were first focused on Navier-Stokes equations, and started from the break-through paper \cite{JiaSve} by Jia and Sverak. It was then extended to DSS solutions \cite{Tsai2013,BT1,BT2,LR16,ChaeWolf2,BT-APDE19}, 
to the half space setting \cite{Korobkov-Tsai,BT2,MR4444075}, and to 
lower regularity data \cite{BT1,BT3,MR3916974}.
In addition to Navier-Stokes equations, 
it has also been extended to fractional Navier-Stokes equations \cite{MR3975031}, 
MHD equations and fractional MHD \cite{MR3975833,MR3997556,MR4205085},
the viscoelastic Navier–Stokes equations with damping \cite{MR3975833}, and more recently the Oberbeck-Boussinesq system \cite{BK}.

For the Oberbeck-Boussinesq system \eqref{OB1}-\eqref{OB2},
large self-similar solutions have been constructed by Brandolese and Karch \cite{BK}, for $v_0,\th_0 \in L^\infty_\loc(\R^3 \setminus \{0\})$. The paper \cite{BK} uses the invading domain method and contradiction arguments, following the approach of Korobkov and Tsai \cite{Korobkov-Tsai}. Solutions are first constructed in finite domains $\Om$ using a priori bounds that may depend on $\Om$, then constructed in $\R^3$ as limits of solutions in $B_k$, $k \to \infty$, using an a priori bound that does not depend on $k$. The key steps are to prove these  
a priori bounds by contradiction arguments. As such, these bounds exist but their values are not explicit.

In this paper we will show that, in fact, for the associated system of \eqref{OB1} in similarity variables, we can find an \emph{explicit a priori bound} for the deviation from a suitably revised background profile, in the same way as in Bradshaw and Tsai \cite{BT1}, (see also \cite{BT2}). This explicit a priori bound not only allows us to construct self-similar solutions, but also allows us to construct discretely self-similar solutions for arbitrary DSS factor $\la>1$. This approach also allows us to relax the assumption on the initial data, and we only need to assume $u_0,\th_0$ are DSS and in $L^{3,\infty}(\R^3)$ (i.e., weak $L^3$). We denote by $L^{q,r}$ Lorentz spaces (\cite{AdamsFournier}).

The method in \cite{BT1} has been extended by Lai \cite{MR3975833} to the
MHD equations and the viscoelastic Navier–Stokes equations with damping. These are coupled systems of Navier–Stokes equations and the equations of other unknowns, and hence are similar in structure to the Oberbeck-Boussinesq system \eqref{OB1}. The a priori bounds in \cite{MR3975833} are slightly different from that in this paper: In  \cite{MR3975833}, the energy estimates of (deviations of) $v$ and other unknowns are suitably combined to cancel trilinear terms. In this paper, the energy estimates of (deviations of) $v$ and $\th$ do not contain trilinear terms, but they need to be suitably combined to absorb the term from $\th \nb G$ to the left side. We will propose a forced MHD system \eqref{MHDG} in Section \ref{sec4} which has an interesting feature that both trilinear terms and large quadratic terms appear in the energy estimates. It is unclear whether we can find an explicit a priori bound for the system, and whether we may prove an implicit a priori bound for it by a contradiction argument. 

We will construct solutions with the following property.
A solution triplet $(v,\th,p)$ of \eqref{OB1} is said to satisfy the \emph{local energy inequalities} if for all non-negative $\phi\in C_c^\infty (\R^3 \times (0,\infty))$, we have 
\EQS{\label{localEnergyIneq}
2\iint |\nabla v|^2\phi\,dx\,dt \leq& \iint |v|^2(\partial_t \phi + \Delta\phi )\,dx\,dt +\iint (|v|^2+2p)(v\cdot \nabla\phi)\,dx\,dt
\\
& + \iint (\th \nb G + f)\cdot v\phi\,dx\,dt,
\\
2\iint |\nabla \th|^2\phi\,dx\,dt \leq& \iint |\th|^2(\partial_t \phi + \Delta\phi )\,dx\,dt +\iint |\th|^2(v\cdot \nabla\phi)\,dx\,dt.
}
They are called ``local energy equalities'' if the inequality sign ``$\le$'' is replaced by equality ``$=$''.
The presence of the local energy inequalities enables one to do local energy estimates (essential for regularity theory), and is a property that is not known, e.g., for all Leray-Hopf weak solutions of the Navier-Stokes equations. 

Note that the two inequalities in \eqref{localEnergyIneq} hold separately. In contrast, for the
MHD equations and the viscoelastic Navier–Stokes equations with damping, studied in Lai \cite{MR3975833}, 
the nonlinear couplings are very strong and only a combined local energy inequality that uses the cancelation between nonlinear couplings is expected hold. See \cite[(1.19), (1.23)]{MR3975833}.
This is related to what we mentioned earlier on the a priori bounds in \cite{MR3975833}.
 
Let $e^{t\Delta}v_0(x)=\int_{\R^3} (4\pi t)^{-3/2}e^{-|x-z|^2/t}v_0(z)\,dz$; this is the solution to the homogeneous heat equation in $\R^3$. 
The main objective of this paper is to prove the following theorem.

\begin{thm}[DSS solution]\label{th:main}
Let $v_0,\th_0$ be a divergence free vector field and a function
in $\R^3$ which are $\lambda$-DSS for some $\lambda >1$ and satisfy 
\begin{equation}\label{ineq:decayingdata}
\|v_0\|_{L^{3,\infty}(\R^3)} + \|\th_0\|_{L^{3,\infty}(\R^3)} \leq c_0,
\end{equation} for a possibly large constant $c_0$. 
Assume $f$ is $\lambda$-DSS, i.e., $f(x,t)=\la^3 f(\la x, \la^2 t)$, 
with $f\in L^\infty(1,\la^2; (L^{6/5,2} + L^2)(\R^3))$.
Then, there exists a $\lambda$-DSS distributional solution $(v,\th,p)$ to \eqref{OB1} with $(v,\th)\in L^2_{\loc}(\R^3\times [0,\infty);\R^3 \times \R)$, $p\in L^{3/2}_{\loc}(\R^3\times [0,\infty))$, 
\EQ{
v(t)-e^{t\Delta}v_0 ,\, \th(t)-e^{t\Delta}\th_0 \in L^\infty(1,\la^2;L^2(\R^3)) \cap L^2(1,\la^2;H^1(\R^3)) ,
}
the local energy inequalities \eqref{localEnergyIneq},  and
\begin{equation}\label{th:main-est}
 \|  v(t)-e^{t\Delta}v_0 \|_{L^2(\R^3)} + \|  \th(t)-e^{t\Delta}\th_0 \|_{L^2(\R^3)}\leq C_0\,t^{1/4}
\end{equation}
for any $t\in (0,\infty)$ and a constant $C_0=C_0(v_0,\th_0)$.
\end{thm}

Above, $L^{q,r}$ denotes Lorentz spaces. By the imbedding $H^1(\R^3) \subset L^{6,2}(\R^3)$ (see e.g.~\cite[Remark 7.29]{AdamsFournier}), we have $(L^{6/5,2} + L^2)(\R^3) \subset H^{-1}(\R^3)$.

Note that we do not require $v_0,\th_0\in L^\infty_\loc(\R^3\setminus \{ 0\})$ as in \cite{BK}. We only need  
$v_0,\th_0\in L^{3,\infty}(\R^3)$. 
When they are $\la$-DSS, being in $L^{3,\infty}(\R^3)$ is equivalent to being in $L^3_\loc(\R^3\setminus \{ 0\})$, as shown in \cite[Lemma 3.1]{BT1}.

\medskip
In the special case that $(v_0,\th_0,f)$ is self-similar, the solution in Theorem \ref{th:main} can also be self-similar, as stated in the following theorem.

\begin{thm}[Self-similar solution]
\label{th:selfsimilardata}
Let $(v_0,\th_0)$ be a divergence free vector field and a function
in $\R^3$ which are $(-1)$-homogeneous and satisfy \eqref{ineq:decayingdata}
for a possibly large constant $c_0$. 
Assume $f$ is $(-3)$-homogeneous with $f(\cdot,1)\in (L^{6/5,2} + L^2)(\R^3)$.
Then, there exists a self-similar distributional solution $(v,\th,p)$ to \eqref{OB1} which satisfies \eqref{th:main-est}
for any $t\in (0,\infty)$ and a constant $C_0=C_0(v_0,\th_0)$.
\end{thm}
The same comment on $v_0,\th_0$ for Theorem \ref{th:main} is still valid for Theorem \ref{th:selfsimilardata}.

As mentioned previously, the proofs of  Theorems \ref{th:main} and \ref{th:selfsimilardata} are based on an explicit a priori bound for the deviation from a suitably revised background profile,
for the system in similarity variables.

The rest of the paper is organized as follows: 
In Section \ref{sec2} we introduce the similarity transform and give properties of the revised profile for the system in similarity variables.
In Section \ref{sec3} we construct discretely self-similar solutions for Theorem \ref{th:main} and self-similar solutions for Theorem \ref{th:selfsimilardata}.
In Section \ref{sec4} we propose a system which does not seem to have an explicit a priori bound. 

\section{Similarity transform and revised profile}\label{sec2}
In this section we first introduce the similarity transform for \eqref{OB1}, following \cite{BT1} for Navier-Stokes equations. We then study properties of a suitably revised profile in similarity variables, and finally show
an explicit a priori bound for the deviation from the revised profile.

\subsection{Similarity transform}
Introduce the similarity transform
\EQS{\label{similarity-transform}
v(x,t)&=\frac1{\sqrt t}\,V(y,s), \quad
\th(x,t)=\frac1{\sqrt t}\,\Th(y,s), \\
p(x,t)&=\frac1{ t}\,P(y,s),\quad \hspace{7pt} f(x,t)=\frac1{\sqrt t^3}\,F(y,s),
}
with
the similarity variables
\EQ{\label{similarity-variables}
y=\frac x {\sqrt t},\quad s = \log t.
}
Note that $G(x)=\frac 1{\sqrt t}G(y)$.
The system \eqref{OB1} becomes
\EQS{\label{Leray}
LV +V\cdot\nabla V+\nabla P &= \Th \nb G+F,\\
\nabla \cdot V&=0,\\
L \Th +V\cdot\nabla \Th& = 0,
}
where the linear operator
\[
L=\partial_s -\Delta  - \frac 12   - \frac 12  y\cdot \nb  .
\]

When $(v,\th)$ is self-similar, the $(V,\Th,F)$ is stationary, i.e., independent of $s$.
When $(v,\th)$ is $\la$-DSS, the $(V,\Th,F)$ is periodic in $s$ with period $T=2\log \la$.
The assumption of Theorem \ref{th:main} implies that $F(y,s)$ is $T$-periodic and
\EQ{ \label{F-cond}
F \in L^\infty(0,T; (L^{6/5,2} + L^2)(\R^3))
\subset L^\infty(0,T; H^{-1}(\R^3)).
}

\subsection{Spatial profile and revised profile}

The heat solutions with initial data $v_0,\th_0$ are transformed accordingly,
\EQ{\label{U0def}
(e^{t\De} v_0) (x)=\frac1{\sqrt t}\, V_0(y,s), \quad
(e^{t\De} \th_0) (x)=\frac1{\sqrt t} \,\Th_0(y,s).
}
The initial condition \eqref{OB2} for $(v,\th)$ implies that $(V_0,\Th_0)$ is the asymptotic spatial profile for $(V,\Th)$ as $|y| \to \infty$. As observed in \cite{BT1}, this can be realized by imposing
\[
\left \{\begin{aligned} 
V-V_0 , \Th - \Th_0&\in L^2(\R^3)  \hspace{24.5mm} \text{in the self-similar case},\\[1pt]
V-V_0 , \Th - \Th_0&\in L^\infty(0,T;L^2(\R^3)) \qquad \text{in the $\la$-DSS case}.
\end{aligned}\right.
\]

\begin{lem}[Profile properties]\label{profile}
Assume $v_0,\th_0 \in L^{3,\infty}(\R^3)$ is DSS with factor $\la>1$. The profile $(V_0,\Th_0)$ defined by \eqref{U0def} is continuously differentiable in $y$ and $s$, and periodic in $s$ with period $T=2\log \la$. They satisfy $LV_0=0$, $\div V_0=0$,  $L\Th_0=0$, and for any $q \in (3,\infty]$,
\EQN{
&V_0,\Th_0 \in L^\infty(0,T; L^{3,\infty} \cap L^q\cap L^\infty(\R^3)),\\
&\pd_s V_0,\pd_s \Th_0 ,
\nb V_0,\nb \Th_0 \in L^\infty(B_R\times (0,T)),\quad \forall 0<R<\infty,
}
and
\[
\sup_{s\in [0,T]} \norm{|V_0|+|\Th_0|}_{L^q(\R^3 \setminus B_R)} \le \mu(R),
\]
for some $\mu:[0,\infty) \to [0,\infty)$ depending on $v_0,\th_0,q$ such that $\mu(R) \to 0$ as $R \to \infty$.
\end{lem}
This is \cite[Lemmas 3.2 \& 3.4]{BT1}. That $V_0\in L^\infty(0,T; L^{3,\infty}(\R^3))$ is not stated in \cite{BT1} but follows directly from $\norm{e^{t\De} v_0}_{L^\infty(0,\infty;L^{3,\infty})} \lec \norm{v_0}_{L^{3,\infty}}$.

To make sense of local energy equality we need $p\in L^{3/2}_{x,t,\loc}$, which requires both $V-V_0$ and $V_0$ in $L^3_{x,t,\loc}$. 
As $V-V_0$ will be sought in the energy class which imbeds into %
$L^{10/3}(\R^3\times (0,T))$, it is convenient to take $q = 10/3$.

Next we will revise the profile $(V_0,\Th_0)$ by essentially
removing its mass in a sufficiently large ball. This does not change its spatial decay, so the revised profile is still large in $L^{3,\infty}$. However,  its $L^q$-norm can be made arbitrarily small for any $q>3$. 

Fix $Z \in C^\infty(\R^3)$ with $0 \le Z \le 1$, $Z(y)=1$ for $|y|>1$ and $Z(y)=0$ for $|y|<1/2$. 
Let $\xi(y)=Z(y/R_0)$ for $R_0$ sufficiently large, and
\[
V_*(y,s) = \xi(y) V_0(y,s) + w, \quad w(y,s)= \nb_y \int_{\R^3}\frac 1{4\pi|y-z|} \nb  \xi(z) \cdot V_0(z,s)\,dz,
\]
\[
\Th_*(y,s) = \xi(y) \Th_0(y,s) .
\]
The correction term $w(y,s)$ is to make $\div V_*=0$.

\begin{lem}[Revised profile]\label{rev-profile}
For any $\al >0$ and $q>3$, there is $R_0=R_0(\al,q,U_0,\Th_0)$ so that $V_*$ and $\Th_*$ given above are continuously differentiable in $y$ and $s$, $T$-periodic, $\div V_*=0$, 
\EQ{\label{lem22a}
V_* - V_0 , \Th_* - \Th_0 \in L^\infty(0,T;  L^2(\R^3)) \cap L^2(0,T;  H^1(\R^3)),
}
\EQ{\label{lem22b}
\norm{|V_*|+|\Th_*|} _{L^\infty(0,T;  L^q(\R^3))} \le \al,
\qquad
\norm{|V_*|+|\Th_*|} _{L^\infty(0,T;  L^4(\R^3))} \le C,
}
and
\EQ{\label{lem22c}
 \norm{L V_*} _{L^\infty(0,T;  L^2(\R^3))}+ \norm{L \Th_*} _{L^\infty(0,T;  L^2 (\R^3))}    \le C,
}
with $C=C(R_0,U_0,\Th_0)$.
\end{lem}

The statements for $V_*$ follows from \cite[Lemma 2.5]{BT1} 
except that, for \eqref{lem22c}, \cite[Lemma 2.5]{BT1} only states
$\norm{L V_*} _{L^\infty(0,T;  H^{-1} (\R^3))}<C$. Tracking its proof, the only missing piece for \eqref{lem22c} is $ \norm{\De w} _{L^\infty(0,T;  L^2(\R^3))}\le C$. This however follows from
\[
\De w = -\nb \div (\nb \xi \cdot V_0)
\]
and the right side is bounded with compact support.
The statement for $\Th_*$ 
follows from 
a similar but easier proof as it has no correction term.

\subsection{A priori bound for the deviation}

Consider $V=V_*+U$ and $\Th=\Th_*+\Psi$. By \eqref{Leray}, the deviation $(U,\Psi)$ satisfies
\begin{align}\label{Leray1}
LU +(V_*+U)\cdot\nabla (V_*+U)+\nabla P &= (\Th_*+\Psi) \nb G+F- LV_*,\\
\nabla \cdot U&=0,\label{Leray2}\\
L \Psi +(V_*+U)\cdot\nabla (\Th_*+\Psi)& = - L \Th_*.\label{Leray3}
\end{align}

We now give a formal argument to derive an a priori bound for the deviation $(U,\Psi)$, which will lead to the a priori bound \eqref{ineq:kenergyevolution} in the proof below for the approximation solutions. Multiply \eqref{Leray1} by $U$ and \eqref{Leray3} by $\Psi$ and integrate by parts over $\R^3 $ using \eqref{Leray2}, we get
\EQS{
\frac d{2ds} \int |U|^2 +  \int \bke{|\nb U|^2 + \frac 14 |U|^2} &= - \int [(V_*+\uwave{U)\cdot \nabla] V_* \cdot U}  - \int (L V_*)\cdot U\label{Leray4}
\\
&\quad + \int [ F+(\Th_*+\uwave{\Psi) \nb G]\cdot U},
}
\EQ{
\frac d{2ds} \int |\Psi|^2 +  \int \bke{|\nb \Psi|^2 + \frac 14 |\Psi|^2} = - \int (V_*+\uwave{U)\cdot(\nabla \Th_*) \Psi } - \int (L \Th_*)\Psi.\label{Leray5}
}
As cubic terms in $(U, \Psi)$ have vanished, the leading terms on the right sides are 3 quadratic terms, which are wavy underlined.
Restrict $3<q<4$ and let $m$ be defined by
\[
\frac 1q+\frac 1m=\frac 12, \quad 4<m<6.
\]
In \eqref{Leray4},
\EQ{\label{2.14}
- \int U\cdot(\nabla V_*)\cdot U  =   \int U\cdot(\nabla U)\cdot V_* \le \norm{\nb U}_{L^2}
\norm{ U}_{L^m}\norm{ V_*}_{L^q}
\le C \alpha  \norm{U}_{H^1}^2
}
by Lemma \ref{rev-profile},
and by Hardy's inequality,
\EQ{\label{2.15}
\int \Psi \nb G\cdot U \le C \norm{\nb \Psi}_{L^2} \norm{\nb U}_{L^2} \le \frac 14 \norm{\nb U}_{L^2}^2+C\norm{\nb \Psi}_{L^2} ^2.
}
In \eqref{Leray5}, by Lemma \ref{rev-profile} again,
\EQ{\label{2.16}
- \int U\cdot(\nabla \Th_*) \Psi  =   \int U\cdot(\nabla \Psi )\Th_* \le\norm{\nb \Psi}_{L^2}
\norm{ U}_{L^m}\norm{ \Th_*}_{L^q}\le C \alpha   \norm{U}_{H^1}\,  \norm{\Psi}_{H^1}.
}

Thus, for $a=\int |U|^2 $, $A=\int |U|^2 + |\nb U|^2$, $b=\int |\Psi|^2 $, $B=\int |\Psi|^2 + |\nb \Psi|^2$, and $\alpha >0$ sufficiently small, we have inequalities of the form
\EQ{\label{2.11}
a' + A \le C + CB, \quad b' + B \le C + C \alpha (A+B),
}
and we can show
\EQ{\label{2.12}
\sup_s (a(s)+b(s)) + \int_0^T (A+B) ds \le C.
}
See the proof of \eqref{ineq:kenergyevolution} for details.

\section{Construction of solutions}\label{sec3}
In this section we will first use the a priori bounds \eqref{2.11}--\eqref{2.12} and the Galerkin method to construct a periodic weak solution of \eqref{Leray1}--\eqref{Leray3} in the class
\EQ{
U, \Psi \in L^\infty(0,T; L^2(\R^3)) \cap L^2(0,T; H^1(\R^3)).
}
We will consider a mollified version of \eqref{Leray1}--\eqref{Leray3}, similar to \cite[(2.23)]{BT1}, so that we can establish local energy inequalities \eqref{localEnergyIneq} for the solution. 
Fix $\eta \in C^\infty_c(\R^3)$, supported in $B_1$ and satisfying $\int_{\R^3}\eta(y)\, dy= 1$. 
For $0<\e<1$, let $\eta_\e (y) = \e^{-3} \eta(y/\e)$ and consider
\begin{align}\label{mLeray1}
LU +(V_*+\eta_\e * U)\cdot\nabla (V_*+U)+\nabla P &= (\Th_*+\Psi) \nb G+F- LV_*,\\
\nabla \cdot U&=0,\label{mLeray2}\\
L \Psi +(V_*+\eta_\e * U)\cdot\nabla (\Th_*+\Psi)& = - L \Th_*.\label{mLeray3}
\end{align}

We now choose a Galerkin basis:
Choose vector $\{\phi_k \}_{k \in \NN} \subset C^\infty_{c,\si}(\R^3)$ which is orthonormal in $L^2_\si(\R^3)$ and whose linear span is dense in $H^1_\si(\R^3)$. 
Choose scalar $\{\beta_k \}_{k \in \NN} \subset C^\infty_{c}(\R^3)$ which is orthonormal in $L^2(\R^3)$ and whose linear span is dense in $H^1(\R^3)$. 

For a fixed $k\in\NN$, 
we look for an approximation solution of \eqref{mLeray1}--\eqref{mLeray3} in a finite dimensional subspace of the form
\EQ{\label{Vk-def}
U_k(y,s) = \sum_{i=1}^k b_{i}(s)\phi_i(y), \quad
\Psi_k(y,s) = \sum_{i=1}^k q_{i}(s)\beta_i(y).
}
Here we omit the dependence of $b_i$ and $q_i$ on $k$.
We first prove the existence of and \emph{a priori} bounds for $T$-periodic solutions $\vec b_{(k)}=(b_{1},\ldots,b_{k})$ and $\vec q_{(k)}=(q_{1},\ldots,q_{k})$ to the system of ODEs
\EQS{\label{eq:ODE}
\frac d {ds} b_{j} = & \sum_{i=1}^k \bke{A_{ij}b_{i} + B_{ij} q_i}+\sum_{i,l=1}^k C_{ilj} b_{i}b_{l} +D_j,\\
\frac d {ds} q_{j} = & \sum_{i=1}^k \bke{A_{ij} ^*b_{i} + B_{ij}^* q_i}+\sum_{i,l=1}^k C_{ilj}^* b_{i}q_{l} +D_j^*,
}
for $j\in \{1,\ldots,k\}$,
where
\begin{align}
\notag A_{ij}&=- (\nabla \phi_i,\nabla \phi_j) 
		+(\tfrac12 \phi_i+\tfrac12 y\cdot \nabla \phi_i  - V_*\cdot  \nabla \phi_i , \phi_j)  + (\eta_\epsilon *\phi_i\otimes V_*, \nb \phi_j)\\
B_{ij} &= (\beta_i\nb G, \phi_j)\\		
\notag C_{ilj}&=- ((\eta_\epsilon *\phi_i \cdot\nabla )\phi_l, \phi_j)
\\\notag D_j&=\langle \cR_b \, ,\phi_j\rangle, \quad \cR_b = - \div (V_* \otimes V_*) +  \Th_* \nb G+F- LV_*,
\end{align}
and 
\begin{align}
\notag A_{ij}^*&=-(\eta_\epsilon *\phi_i\cdot \nabla  \Th_*, \beta_j) \\
B_{ij}^* &= - (\nabla \beta_i ,\nabla \beta_j) 
		+(\tfrac12 \beta_i+\tfrac12 y\cdot \nabla \beta_i -V_*\cdot  \nabla \beta_i , \beta_j) \\		
\notag C_{ilj}^*&=- ((\eta_\epsilon *\phi_i \cdot\nabla) \beta_l, \beta_j)
\\\notag D_j^*&=\langle  \cR_q \, ,\beta_j\rangle, \quad
 \cR_q=- \div(V_* \Th_* )- L\Th_*.
\end{align}
Above 
By \eqref{F-cond}, Lemma \ref{rev-profile} and Hardy's inequality, $\cR_b$ and $\cR_q$ belong to $L^\infty(0,T;H^{-1}(\R^3))$, and all these coefficients are bounded functions of $s$.
(This is slightly better, but unused, than \cite{BT1} where we only said they are $L^2$ in $s$.)

\begin{lem}[Galerkin approximations]\label{lemma:Galerkin} 
Assume the same assumptions on $v_0,\th_0$ in Theorem \ref{th:main}. There is a sufficiently small $\al>0$ such that the following hold. 
\begin{enumerate}
\item For any $k\in \mathbb N$ and $0<\epsilon<1$, the system of ODEs \eqref{eq:ODE} has a $T$-periodic solution $\vec b_{(k)}=(b_{1},\ldots,b_{k}),\vec q_{(k)}=(q_{1},\ldots,q_{k}) \in H^1(0,T)$.
\item Let $U_k(y,s)$ and $\Psi_k(y,s)$ be defined by \eqref{Vk-def}.
We have 
\begin{equation}\label{ineq:uniformink}
||U_k,\Psi_k||_{L^\infty (0,T;L^2(\R^3))} + ||U_k,\Psi_k||_{L^2(0,T;H^1(\R^3))}<C,
\end{equation}
where $C$ is independent of both $\epsilon$ and $k$.
\end{enumerate}
\end{lem}

The smallness of $\al$ will be decided by \eqref{3.13} and \eqref{ineq:kenergyevolution}.

\begin{proof}
Fix $k\in \N$. Omit the subscript $k$ and denote $U=U_k(y,s)$ and $\Psi=\Psi_k(y,s)$. For any  $U^{0}\in \operatorname{span}(\phi_1,\ldots,\phi_k)$ and $\Psi^{0}\in \operatorname{span}(\beta_1,\ldots,\beta_k)$, 
there exist $b_{j}(s), q_j(s)$ uniquely solving \eqref{eq:ODE} with initial value $b_{j}(0)=(U^{0},\phi_j)$, $q_{j}(0)=(\Psi^{0},\beta_j)$, and belonging to $H^1(0,\tilde T)$ for some time $0<\tilde T\leq T$.  If $\tilde T<T$ assume it is maximal -- i.e.~$|\vec b_{(k)}(s)|+|\vec q_{(k)}(s)|\to\infty$ as $s\to \tilde T^-$.

Multiply the $j$-th equations of \eqref{eq:ODE} by $b_{j}$ and $q_j$ and sum. Using the cancelations
\EQS{\label{cancellation}
\tsum_{i,l,j=1}^k C_{ilj} b_{i}b_{l} b_j &= - \bke{(\eta_\epsilon *U \cdot\nabla) U, U}=0,
\qquad (-V_*\cdot \nb U,U)=0,\\
\tsum_{i,l,j=1}^k C_{ilj}^* b_{i}q_{l}q_j &= - \bke{(\eta_\epsilon *U \cdot\nabla) \Psi, \Psi}=0,
\qquad (-V_*\cdot \nb \Psi,\Psi)=0,
}
we obtain 
\begin{align} \label{ineq:1}
\frac 1 2 \frac d {ds} ||U||_{L^2}^2 + \frac 1 4 ||U||_{L^2}^2+ ||\nabla U||_{L^2}^2&=- ( (\eta_\epsilon *U \cdot\nabla)  V_* + \Psi \nb G, U ) + \langle \cR_b, U\rangle,
\\
\frac 1 2 \frac d {ds} ||\Psi||_{L^2}^2 + \frac 1 4 ||\Psi||_{L^2}^2+ ||\nabla \Psi||_{L^2}^2&=- ( (\eta_\epsilon *U \cdot\nabla)  \Th_* , \Psi ) + \langle \mathcal{R}_q, \Psi\rangle. 
\end{align}
These two equalities for $(U,\Psi)=(U_k,\Psi_k)$ correspond to \eqref{Leray4}-\eqref{Leray5}.
Recall
\[
\cR_b = - \div (V_* \otimes V_*) +  \Th_* \nb G+F- LV_*,
 \quad
 \cR_q=- \div (V_*  \Th_* )- L\Th_*.
\]
Thus, by Hardy's inequality
\begin{align}
|\bka{\cR_b,U}| &\le C (\|V_*  \|_{L^4}^2+\|\nb \Th_*  \|_{L^2}+ \norm{F}_{H^{-1}}+ \norm{LV_*}_{H^{-1}} ) \norm{U}_{H^1}  \le C_0+ \frac 1 8 ||U||_{H^1}^2 ,\notag
\\
|\bka{\cR_q,\Psi}| &\le C (\|V_*  \|_{L^4}\|\Th_*  \|_{L^4}+ \norm{L\Th_*}_{H^{-1}} ) \norm{\Psi}_{H^1}  \le C_0+ \frac 1 8 ||\Psi ||_{H^1}^2 ,\label{ineq:2}
\end{align}
where $C_{0}=C(\|V_*  \|_{L^4}^2+\|\Th_*  \|_{L^4}^2+ \|\nb \Th_*  \|_{L^2}+ \norm{F}_{H^{-1}}+ \norm{LV_*}_{H^{-1}} + \norm{L\Th_*}_{H^{-1}} ) ^2 $ is independent of $s$, $T$, $k$, and $\epsilon$.

By Lemma \ref{rev-profile} and Hardy's inequality, (compare \eqref{2.14}--\eqref{2.16})
\begin{align}
\big| ( (\eta_\epsilon *U \cdot\nabla)  V_*, U )  \big| &\leq  C\al ||U||_{H^1}^2 , \label{3.13}\\
\big| ( (\Psi\nb G, V )  \big| &\leq  C ||\nb U||_{L^2}  ||\nb \Psi||_{L^2} ,\\
\big| ( (\eta_\epsilon *U \cdot\nabla)  \Th_*, \Psi )  \big| &\leq  C\al ||U||_{H^1}  ||\Psi||_{H^1} .\label{3.15}
\end{align}
The estimates \eqref{ineq:1}--\eqref{3.15} imply for sufficiently small $\al>0$,
\begin{align}
\frac d {ds} ||U||_{L^2}^2 + \frac 1 4 ||U||_{L^2}^2+ ||\nabla U||_{L^2}^2&\le C_2+C_1  ||\nb \Psi||_{L^2}^2 \label{3.16}
\\
 \frac d {ds} ||\Psi||_{L^2}^2 + \frac 1 4 ||\Psi||_{L^2}^2+ ||\nabla \Psi||_{L^2}^2&\le C_2+C _1\al^2 ||U||_{H^1}  ^2 , \label{3.17}
\end{align}
for some constants $C_1,C_2$.
Multiply the second equation by $2C_1$ and add the first equation. We get for 
\[
f(s)= ||U(s)||_{L^2}^2
+ 2C_1||\Psi(s)||_{L^2}^2 = \tsum_{i=1}^k (b_i^2 + 2C_1 q_i^2)
\]
that
\EQ{
\frac d {ds} f(s) + \frac 14 f + ||\nabla U||_{L^2}^2 + C_1 ||\nabla \Psi||_{L^2}^2 \le C_3 + 2C_1^2 \al^2 ||U||_{H^1}  ^2 ,
}
where $C_3=C_2(1+2C_1) $.
Taking $\al$ sufficiently small so that $2C_1^2 \al^2 \le 1/8$, we get
\EQ{\label{ineq:kenergyevolution} 
\frac d {ds} f(s) + \frac 18 f + \frac12 ||\nabla U||_{L^2}^2 + C_1 ||\nabla \Psi||_{L^2}^2 \le C_3  .
}

The Gronwall's inequality implies
\begin{equation} \label{ineq:gronwall}
e^{s/8} f(s)
\leq f(0) +  \int_0^{\tilde T} e^{\tau/8}  C_3 \,dt
 \le  f(0) + e^{T/8} C_3 T
\end{equation}
for all $s\in [0,\tilde T]$. Since the right hand side is finite, $\tilde T$ is not a blow-up time and we conclude that $\tilde T=T$.  

By \eqref{ineq:gronwall} we can choose sufficiently large $\rho>0$ (independent of $k$) so that 
\begin{equation}\notag
f(0)\leq \rho^2 \Rightarrow f(T) \leq \rho^2.
\end{equation}
Denote by $ B_\rho^{2k}$ the closed ball of radius $\rho$ in $\R^{2k}$ centered at the origin. Let $\mathcal T: B_\rho^{2k}\to B_\rho^{2k}$ 
be the nonlinear map
that maps $(\vec b_{(k)}, \sqrt{2C_1} \vec q_{(k)})(0)$ to $(\vec b_{(k)}, \sqrt{2C_1} \vec q_{(k)})(T)$.   This map is continuous and thus has a fixed point by the Brouwer fixed-point theorem, implying that there exists some $U^{0}\in \operatorname{span}(\phi_1,\ldots,\phi_k)$ 
and $\Psi^{0}\in \operatorname{span}(\beta_1,\ldots,\beta_k)$ 
so that $(\vec b_{(k)}, \sqrt{2C_1} \vec q_{(k)})(0)=(\vec b_{(k)}, \sqrt{2C_1} \vec q_{(k)})(T)$. 
The solution $(\vec b_{(k)}, \vec q_{(k)})(s)$ with this initial data $(\vec b_{(k)},  \vec q_{(k)})(0)$ is $T$-periodic.

It remains to check that \eqref{ineq:uniformink} holds. The $L^\infty L^2$ bound now follows from \eqref{ineq:gronwall}
since $f(0) \le \rho^2$, which is independent of $k$ and $\epsilon$.  
 Integrating  \eqref{ineq:kenergyevolution} in $s \in [0,T]$ and using $f(0)=f(T)$, we get
\begin{equation} \label{eq2.33}
 \int_0^T \big(\frac12 ||\nabla U(s)||_{L^2}^2 + C_1 ||\nabla \Psi(s)||_{L^2}^2
 \big) \,ds
 \le C_3 T,
\end{equation}
which gives an upper bound for $\| U_k, \Psi_k  \|_{L^2(0,T;H^1 )}$ uniform in $k$ and $\epsilon$. 
This finishes the proof of Lemma \ref{lemma:Galerkin}.
\end{proof}

We are now ready to prove Theorem \ref{th:main}.  

\begin{proof}[Proof of Theorem \ref{th:main}] 
We will only sketch the proof since it is similar to the proof of \cite[Theorem 1.2]{BT1} once we have obtained Lemma \ref{lemma:Galerkin} on the existence of $T$-periodic approximation solutions and their uniform estimates.

We have derived from \eqref{OB1} the corresponding system \eqref{Leray} in similarity variables.
Define the profile $(V_0,\Th_0)$ by \eqref{U0def}. 
Under the assumptions of Theorem \ref{th:main}, in particular \eqref{ineq:decayingdata},
the profile $(V_0,\Th_0)$ satisfies Lemma \ref{profile}. Also define the revised profile $(V_*,\Th_*)$ as in Lemma \ref{rev-profile}. The deviation $(U,\Psi)=(V-V_*,\Th-\Th_*)$ satisfies the perturbed system \eqref{Leray1}--\eqref{Leray3}. 

Instead of \eqref{Leray1}--\eqref{Leray3}, considered its mollified version \eqref{mLeray1}--\eqref{mLeray3} with mollifier $\eta_\e$.
As $(V_*,\Th_*)$ satisfies the estimates in Lemma \ref{rev-profile}, Lemma \ref{lemma:Galerkin} gives existence of solutions $(U_{\e,k},\Psi_{\e,k})$ of  \eqref{mLeray1}--\eqref{mLeray3} limited to finite dimensional subspaces with uniform bounds. By a standard limiting argument, there exists $T$-periodic $(U_\e, \Psi_\e)$ satisfying the same bounds \eqref{ineq:uniformink}
and a subsequence of $(U_{\e,k},\Psi_{\e,k})$ (still denoted by $(U_{\e,k},\Psi_{\e,k})$) so that, as $k \to \infty$,
\begin{align*}
& U_{\e,k}\rightarrow U_\epsilon ,\, \Psi_{\e,k}\rightarrow \Psi_\epsilon  \mbox{~weakly in}~L^2(0,T;H^1(\R^3)),
\\&U_{\e,k}\rightarrow U_\epsilon , \,\Psi_{\e,k}\rightarrow \Psi_\epsilon \mbox{~strongly in}~L^2(0,T;L^2(K))  \mbox{~for all compact sets~}K\subset \R^3,
\\& U_{\e,k}(s)\rightarrow U_\epsilon(s) ,\, \Psi_{\e,k}(s)\rightarrow \Psi_\epsilon (s) \mbox{~weakly in}~L^2 (\R^3)\mbox{~for all}~s\in [0,T].
\end{align*}
The weak convergence guarantees that $(U_\e, \Psi_\e)(0)=(U_\e, \Psi_\e)(T)$.  The limit $(U_\e,\Psi_\e)$ is a periodic weak solution of the mollified perturbed Leray system \eqref{mLeray1}--\eqref{mLeray3} satisfying the uniform in $\e$ bound
\begin{equation}\label{ineq:uniforminep}
||U_\e,\Psi_\e||_{L^\infty (0,T;L^2(\R^3))} + ||U_\e,\Psi_\e||_{L^2(0,T;H^1(\R^3))}<C.
\end{equation}

At this stage we construct an associated pressure $P_\epsilon$. %
 Note that $P_\e$ is defined as a distribution whenever $U_\e$ is a weak solution (see \cite{Temam}), but we need to show that $P_\e$ is a function  in $L^{3/2}_{x,t,\loc}$ with a bound uniform in $\e$. 
Rewrite \eqref{mLeray1} as
\[
LU +\nabla P =F_\all ^\e - LV_*,
\quad F_\all^\e =- (V_*+\eta_\e * U_\e)\cdot\nabla (V_*+U_\e)+ (\Th_*+\Psi_\e) \nb G+F.
\]
Note that $ F_\all^\e \in L^\infty(0,T; H^{-1}(\R^3))$.
Taking its divergence and noting $\div  LV_*=0$, we get
\begin{equation}\label{p.eq}
 \Delta P_\epsilon =\div F_{\all}^\e
\end{equation}
in the sense of distributions.
Let 
\begin{equation} \label{tdP.def}
\tilde P_\epsilon = \div\Delta^{-1} F_{\all}^\e.
\end{equation}
It involves Riesz transforms and potentials, and also satisfies \eqref{p.eq}.
We claim that $\nb P_\e=\nb \tilde P_\e$.  To this end we use a well known fact about the forced, non-stationary Stokes system on  $\R^3\times [t_1,t_2]$: If $g\in L^\infty(t_1,t_2;H^{-1}(\R^3))$ and $v_0\in L^2(\R^3)$, then there exists a unique $ v\in  C_w([t_1,t_2];L^2(\R^3))\cap L^2(t_1,t_2;H^1(\R^3))$ and unique distribution $\nabla  p $ satisfying $v(x,t_1)=v_0(x)$ and
\[(\partial_t  v -\Delta  v +\nabla p)(x,t) = g(x,t),\qquad\div v(x,t)=0,\] for $(x,t)\in \R^3\times [t_1,t_2]$. 
For our purpose, let $v(x,t)=t^{-1/2}U_\epsilon(y,s)$,
$ p(x,t) = t^{-1}P_\epsilon(y,s)$, and $g(x,t)=(g_1+g_2)(x,t)$ where $x,y,t,s$ satisfy \eqref{similarity-variables} and
\[
g_1(x,t) = - t^{-3/2} (LV_*)(y,s), \quad g_2(x,t)= t^{-3/2} F_\all^\e(y,s).
\]
Then, $(v,p)$ solves the Stokes system on $\R^3\times [1,\lambda^2]$ with
$g_1,g_2 \in L^\infty(1,\lambda^2;L^2(\R^3))$,  and $v$ is in the energy class.
 We conclude that $\nabla p $ is unique. Since $g \in L^\infty L^2$, $\nabla p$ is given by
 \EQ{
 \nabla p = \nabla (\Delta)^{-1} \div g= \nabla (\Delta)^{-1} \div g_2,
 }
 noting that $g_1$ is divergence free. %
 Since taking Riesz transforms commutes with the above change of variables, we conclude that $\nabla P_\e=\nabla \tilde P_\e$. We may therefore replace $P_\e$ by $\tilde P_\e$,  and apply the Calderon-Zygmund theory to obtain
\begin{equation}\label{ineq:uniforminepP}
\|P_\epsilon \|_{L^{5/3}(\mathbb R^3\times [0,T])}\leq C\|U_\epsilon \|_{L^{10/3}(\mathbb R^3\times [0,T])} ^2 +C\| {V_*} \|_{L^{10/3}(\mathbb R^3\times [0,T])}  ^2,
\end{equation}
which is finite and independent of $\epsilon$ because the uniform energy bound \eqref{ineq:uniforminep} for $U_\e$ and Lemma \ref{rev-profile} with $q=10/3$ for $V_*$. 

Because of the uniform bounds \eqref{ineq:uniforminep} and \eqref{ineq:uniforminepP},
there exists $(U,\Psi,P)$ satisfying the same bounds 
and a subsequence $\e_k$, $k \in \NN$, so that as $\epsilon_k\to 0$,
\begin{align*}
& U_{\epsilon_k} \rightarrow V,\,  \Psi_{\epsilon_k} \rightarrow \Psi \mbox{~weakly in}~L^2(0,T;H^1)
\\& U_{\epsilon_k}\rightarrow V,\, \Psi_{\epsilon_k} \rightarrow \Psi  \mbox{~strongly in}~L^2(0,T;L^2(K)) ~ \forall \mbox{~compact sets $K\subset \R^3$}
\\& U_{\epsilon_k}(s)\rightarrow V(s),\, \Psi_{\epsilon_k}(s) \rightarrow \Psi (s) \mbox{~weakly in}~L^2 \mbox{~for all}~s\in [0,T],\\
&P_{\epsilon_k}\rightarrow P \mbox{~weakly in}~L^{5/3}(\R^3\times [0,T]).
\end{align*}
 Recall that  $(U_\e,\Th_\e,P_\e)$ solves \eqref{mLeray1}--\eqref{mLeray3}. The above convergence is strong enough to ensure that $(U,\Th,P)$ solves the limiting system \eqref{Leray1}--\eqref{Leray3} in the distributional sense. 
  
We then let $V=V_*+U$ and  $\Th=\Th_*+\Psi$, and define $(v,\th,p)$ by 
the similarity transform \eqref{similarity-transform}, which is then a $\la$-DSS solution of \eqref{OB1}. The inequality \eqref{th:main-est} on the convergence to initial data follows from 
\EQ{ \label{conv.v0}
\|  v(t)-e^{t\Delta}v_0 \|_{L^2_x}  = t^{\frac14} \|  (V_*+U)-V_0  \|_{L^2_y},
}
 \eqref{lem22a} and \eqref{ineq:uniforminep}, and similarly for $\th(t)$.

It remains to check that the triplet $(v,\th,p)$ satisfies the local energy inequalities \eqref{localEnergyIneq}. This follows as the approximating solutions $(v_\e,\th_\e, p_\e)$, defined similarly by \eqref{similarity-transform}, all satisfy the \emph{local energy equality}. 
\end{proof} 

\begin{proof}[Proof of Theorem \ref{th:selfsimilardata}] 
We will only sketch the proof since it is similar to the proof of \cite[Theorem 1.3]{BT1}.
For self similar data, we study stationary solutions of the corresponding system \eqref{Leray} of \eqref{OB1} in similarity variables.
Define the profile $(V_0,\Th_0)$ by \eqref{U0def}, and 
the revised profile $(V_*,\Th_*)$ as in Lemma \ref{rev-profile}. They are now independent of $s$. The deviation $(U,\Psi)=(V-V_*,\Th-\Th_*)$ is a stationary solution of the perturbed system \eqref{Leray1}--\eqref{Leray3}. For $k \in \NN$, its Galerkin approximation \eqref{Vk-def}  is a stationary solution
of the ODE system \eqref{eq:ODE}. Because stationary weak solutions are automatically smooth and satisfy the local energy equalities, we
do not need mollification and can ignore $\eta_\e*$ in \eqref{eq:ODE} and the definitions of the coefficients below \eqref{eq:ODE}. 

Let $m = \sqrt{2C_1}$ where $C_1$ is the constant in \eqref{3.16}--\eqref{3.17}.
Denote $x=(x',x'')\in \R^{2k}$ with $x',x''\in \R^k$. %
Define a map $P(x) : \R^{2k} \to \R^{2k}$ by
\EQS{\label{Pxdef}
P(x)_{j} = & \sum_{i=1}^k \bke{A_{ij}b_{i} + B_{ij} q_i}+\sum_{i,l=1}^k C_{ilj} b_{i}b_{l} +D_j,\\
P(x)_{k+j}  = & m\bigg\{\sum_{i=1}^k \bke{A_{ij} ^*b_{i} + B_{ij}^* q_i}+\sum_{i,l=1}^k C_{ilj}^* b_{i}q_{l} +D_j^*\bigg\},
}
for $j\in \{1,\ldots,k\}$, where
\EQ{\label{bq.def}
b=x', \quad q=\frac 1m x'',
}
and the coefficients are as defined below \eqref{eq:ODE}. Let $U = \sum_{i=1}^k b_{i}\phi_i$ and $\Psi = \sum_{i=1}^k q_{i}\beta_i$. Using the cancellation \eqref{cancellation},
we have
\EQN{
P(x) \cdot x&=-
 \frac 1 4 ||U||_{L^2}^2- ||\nabla U||_{L^2}^2- ( (U \cdot\nabla)  V_* + \Psi \nb G, U ) + \langle \cR_b, U\rangle\\
&\quad + m^2\bigg\{ -\frac 1 4 ||\Psi||_{L^2}^2- ||\nabla \Psi||_{L^2}^2- ( (U \cdot\nabla)  \Th_* , \Psi ) + \langle \mathcal{R}_q, \Psi\rangle \bigg\}.
}
For $\al>0$ sufficiently small,  
the estimates \eqref{ineq:2}--\eqref{3.15} imply
\begin{align}
P(x) \cdot x&\le - \frac 1 8 ||U||_{L^2}^2- \frac 12 ||\nabla U||_{L^2}^2
 +C_2+C_1  ||\nb \Psi||_{L^2}^2 \notag
\\
&\quad + m^2\bigg\{ -\frac 1 8 ||\Psi||_{L^2}^2-\frac 12 ||\nabla \Psi||_{L^2}^2 
+ C_2+C _1\al^2 ||U||_{H^1}  ^2 \bigg\} \notag \\
&\le  -  \frac 1 8 \bigg( ||U||_{L^2}^2 +||\nabla U||_{L^2}^2 +  m^2 ||\Psi||_{L^2}^2 +||\nabla \Psi||_{L^2}^2 \bigg) + C_2(1+m^2).\label{3.30}
\end{align}
Note that
\[
||U||_{L^2}^2 + m^2 ||\Psi||_{L^2}^2 = |b|^2 + m^2 |q|^2 = |x|^2 .
\]
Hence $P(x) \cdot x \le 0$ if
\[
|x|= \rho := [ 8C_2(1+m^2)]^{\frac12}.
\]
By a variant of Brouwer’s fixed point theorem \cite[Lemma 2.6]{nslec}, there is one $x$ with $|x| < \rho$ such that $P(x) = 0$. 
For this $x$, \eqref{bq.def} gives a stationary solution of \eqref{eq:ODE}, with a priori bound
\EQ{
 ||U||_{L^2}^2 +||\nabla U||_{L^2}^2 +  m^2 ||\Psi||_{L^2}^2 +||\nabla \Psi||_{L^2}^2 \le 8 C_2(1+m^2)
}
independent of $k$, by inequality \eqref{3.30} and $P(x)=0$.
This bound is sufficient to find a subsequence with a weak limit in $H^1(\R^3)$ and a strong limit in $L^2(K)$ for any compact set $K$ in $\R^3$, that is, there exists a stationary solution $(U,\Psi)$ to the perturbed system \eqref{Leray1}--\eqref{Leray3}. The sum $(V,\Th)=(V_*+U,\,\Th_*+\Psi)$ is then 
 a stationary weak solution of \eqref{Leray}. A pressure $P$ can be defined by \eqref{tdP.def} using Riesz transforms. The triplet $(v,\th,p)$ given by the similarity transform \eqref{similarity-transform} is then a self-similar solution of \eqref{OB1} satisfying the convergence to initial data inequality \eqref{th:main-est} using \eqref{conv.v0}. This completes the proof of Theorem \ref{th:selfsimilardata}.
 \end{proof}

\section{A forced MHD system}\label{sec4}

In this section we propose a system which has the interesting feature that both trilinear terms and large quadratic terms appear in the energy estimates. Because
it is unclear whether we can find an explicit a priori bound for the system, and whether we may prove an implicit a priori bound for it by a contradiction argument, we do not know whether we can construct self-similar or discretely self-similar solutions for the system.

Consider the following forced MHD system
\EQS{\label{MHDG}
\partial_t v-\Delta v +v\cdot\nabla v - b \cdot \nb b +\nabla p &= f(b),\\
\partial_t b-\Delta b +v\cdot\nabla b - b\cdot \nb v& = 0,\\
\nabla \cdot v=\nabla \cdot b&=0,
}
for the unknown fluid velocity $v(x,t)$, the magnetic field $b(x,t)$ and the pressure $p = p(x,t)$, 
 defined for $x \in \R^3$ and $ t>0$. The term $f(b)$ is a scaling-critical vector field depending linearly in $b$, e.g.,
\EQ{
f(b) = b \times \nb \frac M{|x|},
}
or
\EQ{
f(b) = \frac M{|x|^2}\, b ,
}
for some constant $M$.
If $b$ is transformed as \eqref{similarity-transform}, $b(x,t) = t^{-1/2} B(y,s)$, then $f(b)$ satisfies
\EQ{
f(b(x,t)) = \frac 1{\sqrt t^{3}} \, f(B(y,s)).
}
The system \eqref{MHDG}
is coupled with initial conditions
\EQ{\label{MHDG2}
v(x,0)=v_0(x), \quad b(x,0)=b_0(x).
}

Formal
energy estimates give
\begin{align}
\frac d{2dt} \int |v|^2 +  \int |\nb v|^2  - \int b \cdot \nb b \cdot v &\le \int f(b)\cdot v ,\label{eq4.5}
\\
\frac d{2dt} \int |b|^2 +  \int |\nb b|^2  - \int b \cdot \nb v \cdot b &\le 0.\label{eq4.6}
\end{align}
To remove the nonlinear effect of trilinear terms, we add \eqref{eq4.5} and \eqref{eq4.6} with constant factor 1 for both equations. But then we cannot control the quadratic term $\int f(b)\cdot v $ on the right side for large $M$, as we went from \eqref{3.16}--\eqref{3.17} to \eqref{ineq:kenergyevolution}.

The same issue remains in similarity variables, as we only add $\frac 14 \int |v|^2$ and $\frac 14 \int |b|^2$ to the left sides of \eqref{eq4.5} and \eqref{eq4.6}; compare \eqref{Leray4} and \eqref{Leray5}.

As we cannot control the trilinear terms and the large quadratic term $\int f(b)\cdot v $ at the same time, it is unclear whether we can find an explicit a priori bound for the system. On the other hand, it is also unclear whether we may prove an implicit a priori bound by a contradiction argument. 
This makes this system interesting: Is this a candidate of a system that we may prove existence of self-similar solutions, but not existence of discretely self-similar solutions?

\section*{Acknowledgments}
I warmly thank Zachary Bradshaw and Chen-Chih Lai for helpful comments. The research of TT was partially supported by Natural Sciences and Engineering Research Council of Canada (NSERC) under grant RGPIN-2023-04534.

\bibliographystyle{abbrv}
\bibliography{/Users/ttsai/Nextcloud/tex/bib/fluid2024}

\end{document}